\documentclass{amsart}
\usepackage{amsmath}
\usepackage{amssymb}
\usepackage{amsthm}

\DeclareMathOperator{\Ric}{Ric}

\newcommand{\cc}{\widetilde{c}}
\newcommand{\cg}{\widetilde{g}}
\newcommand{\og}{\overline{g}}
\newcommand{\cf}{\widetilde{f}}
\newcommand{\cm}{\widetilde{m}}
\newcommand{\lp}{\langle}
\newcommand{\rp}{\rangle}
\DeclareMathOperator{\cD}{\widetilde{\Delta}}

\newtheorem{thm}{Theorem}[section]

\newtheorem{lem}[thm]{Lemma}
\newtheorem{cor}[thm]{Corollary}

\theoremstyle{definition}

\theoremstyle{remark}

\numberwithin{equation}{section}

\begin{document}

\title[On the nonexistence of quasi-Einstein metrics]{On the nonexistence of quasi-Einstein metrics}
\author{Jeffrey S. Case}
\thanks{Partially supported by the AAS and NSF grant OISE-0812807}
\address{Department of Mathematics\\University of California\\Santa Barbara, CA 93106}
\email{casej@math.ucsb.edu}
\date{}
\begin{abstract}
We study complete Riemannian manifolds satisfying the equation $\Ric+\nabla^2f-\frac{1}{m}df\otimes df=0$ by studying the associated PDE $\Delta_f f +m\mu \exp(2f/m)=0$ for $\mu\leq 0$.  By developing a gradient estimate for $f$, we show there are no nonconstant solutions.  We then apply this to show that there are no nontrivial Ricci flat warped products with fibers which have nonpositive Einstein constant.  We also show that for nontrivial steady gradient Ricci solitons, the quantity $R+|\nabla f|^2$ is a positive constant.
\end{abstract}
\maketitle

\section{Introduction}

An interesting question posed by Besse \cite{Besse} is that of determining when one can construct examples of Einstein manifolds which are warped products.  If $(M,g)$ and $(N^m,h)$ are Riemannian manifolds, the warped product $(M\times N, \og)$, where $\og=g\oplus \exp(-2f/m)h$, is Einstein if and only if $(N,h)$ is Einstein and
\begin{align}
\label {ec1} \Ric_f^m & = \lambda g \\
\label {ec2} \Delta_f f - m\lambda & = -m\mu \exp\left(\frac{2}{m}f\right) ,
\end{align}
where $\Ric(h)=\mu h$,$\Ric(\og)=\lambda\og$, 
\[ \Ric_f^m=\Ric+\nabla^2f-\frac{1}{m}df\otimes df \]
is the Bakry-\'Emery-Ricci tensor, $\nabla^2$ denotes the Hessian and $\Delta_f u = \Delta u - \nabla f\cdot\nabla u$.  We will call $f$ the potential.  As a result of the Bianchi identity, Kim and Kim~\cite{Kim_Kim} proved that \eqref{ec1} implies \eqref{ec2} for some constant $\mu$, and thus one can study Einstein warped product manifolds by studying only \eqref{ec1} on the base $(M,g)$.

If one takes $m=\infty$ in \eqref{ec1}, one is studying gradient Ricci solitons.  In this case, the Bianchi identity yields 
\begin{equation}
\label{sc1}
\Delta_f f + 2\lambda f = -\mu
\end{equation}
for some constant $\mu$ (cf.\ \cite{Ivey_1994}).  The sign here is chosen so that one can view \eqref{sc1} as the limit $m\to\infty$ of \eqref{ec2}.

In \cite{Case_Shu_Wei}, a quasi-Einstein metric $g$ is defined as a metric such that $\Ric_f^m(g)=\lambda g$ for some constant $\lambda$, where $0<m\leq\infty$.  The observation of Kim and Kim together with its generalization to the $m=\infty$ case allows us to study the nonexistence of nontrivial warped product Einstein metrics by considering only \eqref{ec2} and \eqref{sc1}, where a quasi-Einstein metric is nontrivial if the potential is nonconstant.  This will be the point of view of this paper.

Examples of quasi-Einstein manifolds with $\lambda<0$ and $\mu$ of arbitrary sign are constructed in Besse~\cite{Besse}, as well as examples with $\lambda=0$ and $\mu\geq 0$ in \eqref{ec1} and \eqref{ec2}.  Moreover, in the latter case, all of the nontrivial examples have $\mu>0$, while the trivial quasi-Einstein metric $\Ric_f^m=0$ necessarily satisfies $\mu=0$.  More recently, L\"u, Page and Pope~\cite{LPP} constructed nontrivial quasi-Einstein metrics with $\lambda>0$ and $m>1$, which also satisfy $\mu>0$.  On the other hand, it is known that if $m<\infty$ and $\lambda>0$, then $M$ is necessarily compact (cf.\ \cite{Wei_Wylie_survey}).  Thus, the maximum principle applied to \eqref{ec2} yields that $\mu>0$.  From these results, all that remains to be understood is whether there are nontrivial quasi-Einstein metrics with $\lambda=0$ and $\mu\leq 0$.

For steady gradient Ricci solitons, this question is also interesting.  The Bryant and Ivey steady solitons~ \cite{Ivey_1994} are all examples of nontrivial steady gradient Ricci solitons with $\mu>0$.  Dancer and Wang~\cite{DancerWang2009} later generalized the construction to provide an even larger class of examples of nontrivial steady gradient Ricci solitons, and they too all had $\mu>0$.  Thus one is also led to wonder if there are nontrivial steady gradient Ricci solitons with $\mu\leq 0$.

In this paper, we address this question, and show that indeed there are no such quasi-Einstein metrics.  Specifically, we prove the following:

\begin{thm}
\label{main_thm_simple}
Let $(M,g)$ be a complete Riemannian manifold such that $\Ric_f^m=0$ for some smooth function $f$ and $0\leq m\leq\infty$, and let $\mu$ be the constant given by
\begin{equation}
\label{pde_simple}
\Delta_f f + \mu \exp\left(\frac{2}{m}f\right) = 0 .
\end{equation}
Then $\mu\geq 0$, and equality holds if and only if $(M,g)$ is Ricci flat.
\end{thm}

In fact, we shall actually prove the following slightly stronger theorem:

\begin{thm}
\label{main_thm}
Let $(M,g)$ be a complete Riemannian manifold such that $\Ric_f^m\geq0$ for some smooth function $f$ and $0\leq m\leq\infty$, and suppose that
\begin{equation*}
\label{pde}
\Delta_f f = c_1 \exp(c_2 f)
\end{equation*}
for constants $c_1,c_2\geq 0$.  Then $f$ is constant.
\end{thm}

We first note that in Theorem~\ref{main_thm_simple}, the condition~\eqref{pde_simple} encompasses both \eqref{ec2} and \eqref{sc1}, where the exponential term is understood to be equal to one if $m=\infty$.  We also note that in Theorem~\ref{main_thm_simple} for the case $m=\infty$, the definition of $\mu$ given by~\eqref{pde_simple} is equivalent to the definition
\[ R + |\nabla f|^2 = \mu . \]
Thus, in this case our result would follow immediately from knowing that the scalar curvature $R\geq 0$.  As was pointed out to the author by McKenzie Wang after an early version of this paper was made available, this has been shown in the context of ancient solutions of the Ricci flow by B.-L.\ Chen~\cite{Chen2007}, and explicitly in our setting by Z.\ Zhang~\cite{Zhang2008}.  However, this result does not imply Theorem~\ref{main_thm}.  Theorem~\ref{main_thm} is also of interest in that it is an example of a result whereby one can ``take the limit'' $m\to\infty$, even though the underlying Laplacian comparison result does not extend.

Together with the maximum principle result for quasi-Einstein metrics with $\lambda>0$, Theorem~\ref{main_thm_simple} then yields the following partial answer to the question posed by Besse:

\begin{cor}
Let $M\times N$ be a nontrivial Einstein warped product manifold with nonnegative scalar curvature.  Then at least one of $M$ and $N$ is Einstein with positive scalar curvature.
\end{cor}

Another interesting corollary to our theorem, pointed out to us by Yujen Shu, deals with conformally Einstein Riemannian products.  If $(M,g)$ and $(N,h)$ are Riemannian manifolds, then the manifold $M\times N$ is a conformally Einstein Riemannian product if the standard product metric is conformally Einstein.  Using a characterization of Cleyton \cite{Cleyton} of such manifolds, we have the following corollary.

\begin{cor}
Suppose that $M\times N$ is a complete conformally Ricci flat Riemannian product.  Then either one of $M$ or $N$is Einstein with positive Einstein constant or both $M$ and $N$ are Ricci flat.
\end{cor}

Our result is originally motivated by a result of Anderson~\cite{Anderson_1999} for static metrics in general relativity (cf. \cite[Chapter 6]{Wald_book}).  In that setting, static metrics are Riemannian triples $(M^3,g,u)$, where $u$ is called the static potential, such that
\begin{align*}
u\Ric & = \nabla^2 u\\
\Delta u & = 0 .
\end{align*}
This yields a Ricci-flat spacetime metric $-u^2 dt^2 \oplus g$ on the product $\mathbb{R}\times M$.  In the language of quasi-Einstein metrics, static metrics are then just quasi-Einstein metrics $(M^3,g,f)$ with the constants $\lambda=0=\mu$, where $u=e^{-f}$.  Using PDE methods, Anderson proved that if $(M,g,u)$ is a complete static metric with $u>0$, then $u$ must in fact be constant.  One of our original observations was that, using comparison results for manifolds with $\Ric_f^m\geq 0$ and the Harnack inequality for the $f$-Laplacian of Li~\cite{Li_XD_2005}, his proof generalizes almost immediately to the case when $\lambda=\mu=0$ and $m<\infty$, thus providing the original inspiration to study nonexistence of quasi-Einstein metrics using PDE methods.

In order to extend this result to the steady gradient Ricci soliton case $m=\infty$, as well as to rule out the possibility $\mu<0$, we instead focus on the gradient estimate that leads to the Harnack inequality (cf.\ the Schoen-Yau gradient estimate~\cite[Theorem 3.1]{Schoen_Yau_book}).  More precisely, we arrive at the following estimate:

\begin{thm}
\label{gradient}
Let $(M^n,g,f,m)$ be such that $\Ric_f^m\geq 0$, $m<\infty$, and $\Delta_f f = \phi(f)$, where $\phi\colon\mathbb{R}\to\mathbb{R}$ is such that
\[ \phi^\prime(t) + \frac{2}{n}\phi(t) \geq 0 \]
for all $t\in\mathbb{R}$.  Then for all $x\in M$ and $a>0$ such that $B(x,a)$ is geodesically connected in $M$ and the closure $\overline{B(x,a)}$ is compact,
\[ |\nabla f|^2(x) \leq \frac{2n(m+n+6)}{a^2} . \]
\end{thm}

The crucial aspect of this theorem is the dependence of the gradient estimate on $m$.  Though the gradient estimate itself does not hold when $m=\infty$, we will be able to find a conformally related triple $(M,\cg,\cf)$ which satisfies the hypotheses of Theorem~\ref{gradient} for any $0<\cm<\infty$.  The dependence of the gradient estimate on $m$ will then allow us to prove that $(M,\cg)$ is complete for $\cm$ large enough, which then yields Theorem~\ref{main_thm}.  We also note that there is a natural interpretation for this conformal transformation coming from the relationship between warped product Einstein metrics and conformally Einstein products, which will be discussed in Section~\ref{conformal}.

To fix notation, throughout this paper we will be considering a Riemannian manifold $(M^n,g,f,m)$ which are not necessarily complete.  Unless otherwise stated, we shall also allow $0\leq m\leq \infty$.

The author would like to thank Robert Bartnik and Pengzi Miao for bringing to our attention the connection between static metrics and quasi-Einstein metrics and many interesting discussions, from which the idea for this paper arose.  He would also like to thank Xianzhe Dai, Yujen Shu, and Guofang Wei for helpful discussions on relations with Ricci solitons and comments on early drafts of this paper.  He would also like to thank Monash University for their hospitality while he visited in Summer 2008.  Finally, he would like to thank the reviewer for many helpful comments, and in particular his suggestion that we explicitly state the more general Theorem~\ref{main_thm}.

\section{The Conformal Rescaling}
\label{conformal}

As mentioned in the introduction, we will find it useful to conformally rescale steady gradient Ricci solitons.  The desired rescaling is suggested by the following lemma:

\begin{lem}
\label{lem:change}
Let $(M^n,g)$ be a Riemannian manifold, let $f$ be a smooth function on $M$, and let $0<m<\infty$.  Define the conformally related triple $(M,\cg,\cf)$ by
\begin{align*}
\cg & = \exp\left(-\frac{2}{m+n-2}f\right)g \\
\cf & = \frac{m}{m+n-2}f .
\end{align*}
Then
\begin{align}
\label{eqn:change} \Ric_{\cf}^m(\cg) & = \Ric(g) + \nabla_g^2 f + \frac{1}{m+n-2}df\otimes df + \frac{1}{m+n-2}\Delta_f f\;g \\
\label{eqn:change2} \cD_{\cf} & = \exp\left(\frac{2}{m+n-2}f\right)\Delta_f .
\end{align}
\end{lem}

The proof will be omitted, as it is a straightforward calculation using the well-known formulae for the change of the Ricci curvature and the Hessian of a function under a conformal change of metric, as can be found in~\cite{Besse}.  A closely related expression appears in the recent preprint of D.\ Chen~\cite{Chen2009}, where he constructs conformally Einstein Riemannian products.  Indeed, the choice of conformal change comes from the equivalence of the two metrics
\[ \cg\oplus \exp\left(-\frac{2}{m}\cf\right) h = \exp\left(-\frac{2}{m+n-2}f\right)\left(g\oplus h\right), \]
where the constants are chosen so that the Ricci curvature restricted to horizontal vector fields is given by~\eqref{eqn:change}.  This ``duality'' between warped product metrics and conformal Riemannian products is a useful tool for the study of quasi-Einstein metrics, as will be discussed in a forthcoming paper of the author.

For the purposes of this paper, the important feature of the above lemma is that on the right hand side of \eqref{eqn:change}, the coefficient of the quadratic term is \emph{positive}.  In particular, as
\[ \Ric + \nabla^2 f + \frac{1}{m+n-2}df\otimes df \geq \Ric_f, \]
we can the use Lemma~\ref{lem:change} to pass a lower bound for $\Ric_f$ to a lower bound for $\Ric_f^m$ for a conformally related metric.  More precisely, we arrive at the following corollary:

\begin{cor}
\label{conf_relation}
Suppose $(M^n,g,f)$ satisfies $\Ric_f\geq0$ and $\Delta_f f = c_1\exp(c_2 f)$.  Fix $m<\infty$, and define $(M,\cg,\cf)$ as in Lemma~\ref{lem:change}.  Then
\begin{align*}
\Ric_{\cf}^{m}(\cg) & \geq \frac{c_1}{m+n-2} \exp\left((c_2+\frac{2}{m+n-2})f\right)\cg \\
\cD_{\cf} \cf & = \frac{mc_1}{m+n-2}\exp\left((c_2+\frac{2}{m+n-2})f\right)
\end{align*}
In particular, if $c_1\geq 0$, then
\[ \Ric_{\cf}^{m}(\cg) \geq 0 . \]
\end{cor}

\begin{proof}

The second equality follows immediately from~\eqref{eqn:change2}, and the first inequality follows immediately from~\eqref{eqn:change} and the observation that, under the assumption $\Ric_f\geq 0$,
\[ \Ric_{\cf}^{m}(\cg) \geq \frac{1}{m+n-2}\Delta_f f \; g . \qedhere \]
\end{proof}

\section{Gradient estimate}
\label{comparison}

In this section we establish our gradient estimate for the potential.  As a corollary, we arrive at a Liouville-type theorem which, together with an argument using the conformally related metrics of the previous section, yields Theorem~\ref{main_thm}.  To start, we will need two standard comparison results for the Bakry-\'Emery-Ricci tensor, the Bochner formula and the Laplacian comparison theorem, both of which results can be found, for example, in the survey article \cite{Wei_Wylie_survey}.

\begin{thm}[Bochner formula]
Let $u,f\in C^\infty(M)$.  Then
\[ \frac{1}{2}\Delta_f |\nabla u|^2 = |\nabla^2 u|^2 + \lp\nabla u,\nabla\Delta_f u\rp + \Ric_f^m(\nabla u,\nabla u) + \frac{1}{m}\lp\nabla f,\nabla u\rp^2 . \]
\end{thm}

One can quickly derive from the Bochner formula the Laplacian comparison theorem.  In the following, when we say that $U\subset M$ is geodesically connected in $M$, we mean that for all $p,q\in U$, there is a minimal geodesic connecting $p$ to $q$ which lies in $M$.

\begin{thm}[Laplacian comparison]
Let $U\subset M$ be geodesically connected in $M$ and suppose that $\Ric_f^m\geq 0$ on $M$.  Fix $x\in U$ and let $r$ denote the distance function from $x$.  Then outside the cut locus of $x$ in $U$, we have the estimate
\[ \Delta_f r \leq \frac{m+n-1}{r} . \]
\end{thm}

With these tools, we are then able to prove our gradient estimate:

\begin{proof}[Proof of Theorem~\ref{gradient}]

Applying the Bochner formula to $f$, we see that
\[ \frac{1}{2}\Delta_f|\nabla f|^2 \geq |\nabla^2 f|^2 + \phi^\prime(f)|\nabla f|^2 + \frac{1}{m}|\nabla f|^4 \]
On the other hand, the Cauchy-Schwarz inequality yields
\begin{align*}
|\nabla^2 f|^2 & \geq \frac{1}{n}(\Delta f)^2 \\
& = \frac{1}{n}(\Delta_f f + |\nabla f|^2)^2 \\
& = \frac{1}{n}\phi^2(f) + \frac{2}{n}\phi(f)|\nabla f|^2 + \frac{1}{n}|\nabla f|^4 \\
& \geq \frac{1}{n}|\nabla f|^4 + \frac{2}{n}\phi(f)|\nabla f|^2 .
\end{align*}
Hence, using the assumption on $\phi$, we arrive at the estimate
\[ \Delta_f |\nabla f|^2 \geq \frac{2}{n}|\nabla f|^4 . \]

Now consider the function $F=(a^2-r^2)^2|\nabla f|^2$ defined on $B(x,a)$, with $r(y)=d(x,y)$ the radial distance function.  Then there is a point $x_0$ in the interior of $B(x,a)$ such that $F$ achieves its maximum at $x_0$.  Using the method of support functions if necessary, we may assume that $x_0$ lies outside the cut locus of $x$.  At the point $x_0$, we necessarily have
\begin{align*}
\frac{d|\nabla f|^2}{|\nabla f|^2} & = \frac{2d(r^2)}{a^2-r^2} \\
0 & \geq -\frac{2\Delta_f r^2}{a^2-r^2} + \frac{\Delta_f|\nabla f|^2}{|\nabla f|^2} + 2\frac{|\nabla r^2|^2}{(a^2-r^2)^2} - \frac{4\lp\nabla|\nabla f|^2,\nabla r^2\rp}{(a^2-r^2)|\nabla f|^2} .
\end{align*}
The Laplacian comparison theorem yields $\Delta_f r^2\leq 2(m+n)$, and so combining the above inequalities with the estimate for $\Delta_f|\nabla f|^2$, we see that
\[ 0 \geq \frac{2}{n}|\nabla f|^2 - \frac{4(m+n)}{a^2-r^2} - \frac{24r^2}{(a^2-r^2)^2} . \]
Multiplying through by $(a^2-r^2)^2$, we see that
\[ 0 \geq \frac{2}{n}F - 4(m+n+6)a^2 \]
and so
\[ \sup_{B(x,a)} (a^2-r^2)^2|\nabla f|^2 \leq 2n(m+n+6)a^2 . \]
In particular, $a^4|\nabla f|^2(x) \leq 2n(m+n+6)a^2$.
\end{proof}

In particular, the above theorem can be applied to manifolds satisfying the hypotheses of Theorem~\ref{main_thm} with $m<\infty$ to achieve a Liouville-type theorem:

\begin{cor}
\label{liouville}
Let $(M,g)$ be a complete Riemannian manifold and let $0<m<\infty$ and $f\in C^\infty(M)$ be such that $\Ric_f^m\geq 0$ and $\Delta_f f = c_1\exp(c_2f)$ for $c_1,c_2\geq 0$.  Then $f$ is constant.
\end{cor}

\begin{proof}

The function $\phi(t)=c_1\exp(c_2t)$ clearly satisfies $\phi^\prime+\frac{2}{n}\phi\geq 0$, and so we can apply Theorem \ref{gradient}.  Because $(M,g)$ is complete, we may take $a\to\infty$, which yields the result.
\end{proof}

\section{Proof of Theorem~\ref{main_thm}}

If $m<\infty$, Corollary~\ref{liouville} implies that $f$ is constant.  On the other hand, if $m=\infty$, choose some $\cm<\infty$ and let $(M,\cg,\cf)$ be the conformally related triple defined by
\[ \cg = \exp\left(-\frac{2}{\cm+n-2}f\right)g , \qquad \cf = \frac{\cm}{\cm+n-2}f . \]
Corollary~\ref{conf_relation} then implies that there are constants $\cc_1,\cc_2$ such that
\begin{align*}
\Ric_{\cf}^{\cm}(\cg) & \geq 0 \\
\cD_{\cf}\cf & = \cc_1 \exp(\cc_2 f) .
\end{align*}
If $(M,\cg)$ is complete, then applying Corollary~\ref{liouville} again yields the desired result.  However, a priori we do not know that $(M,\cg)$ is complete, and so we turn to addressing this issue.

Define
\[ R=\sup\{r\colon B(x_0,r)\mbox{ is compactly contained in } M\} . \]
If $R=\infty$, then $(M,\cg)$ is complete, so suppose instead that $R<\infty$.  Thus there is an inextendible unit speed $\cg$-geodesic $\gamma\colon[0,R)\to (M,\cg)$ with $\gamma(0)=x_0$.  Because $(M,g)$ is complete, we know that
\[ L_g(\gamma) = \int_0^R \exp\left(\frac{\cf(\gamma(t))}{\cm}\right) dt = \infty . \]
We will derive a contradiction by showing that $L_g(\gamma)<\infty$.  To that end, let $t\in[0,R)$, and consider the $\cg$-geodesic balls $B_t=B(\gamma(t),(R-t)/2)$.  The triangle inequality implies $B_t$ is geodesically connected in $(M,\cg)$, and so we can apply our gradient estimate to $f$ in $B_t$, yielding
\[ |\tilde\nabla \cf|(\gamma(t)) \leq \frac{C}{R-t} , \]
where $C=\sqrt{8n(\cm+n+6)}$.  In particular, $C\in O(\cm^{1/2})$.  Integrating the above inequality yields
\[ \exp\left(\frac{\cf(\gamma(t))}{\cm}\right) \leq C_1(R-t)^{-C/\cm} \]
for some constant $C_1>0$.  Since $C/\cm\in O(\cm^{-1/2})$, we can choose $\cm$ sufficiently large such that $C/\cm<1$ . Hence $L_g(\gamma)<\infty$, a contradiction, and so we see that $(M,\cg)$ must be complete. \qed

\bibliographystyle{abbrv}
\bibliography{../bib}
\end{document}